\documentclass[11pt]{article}
\usepackage{amssymb,amsmath,amsthm}
\usepackage{epsfig}
\usepackage{psfrag}
\usepackage{cite}
\newtheorem{theorem}{\sc Theorem}[section]
\newtheorem{lemma}[theorem]{\sc Lemma}

\newtheorem{corollary}[theorem]{\sc Corollary}
\newtheorem{definition}[theorem]{\sc Definition}

\newtheorem{remark}[theorem]{\sc Remark}

\def\hpic #1 #2 {\mbox{$\begin{array}[c]{l} 
\epsfig{file=#1,height=#2}\end{array}$}}
\def\wpic #1 #2 {\mbox{$\begin{array}[c]{l} 
\epsfig{file=#1,width=#2}\end{array}$}}

\def\C{\mathbb C}

\def\be{\begin{equation}}
\def\ee{\end{equation}}
\def\bt{\begin{theorem}}
\def\et{\end{theorem}}
\def\bi{\begin{itemize}}
\def\ei{\end{itemize}}
\def\bea{\begin{eqnarray}}
\def\eea{\end{eqnarray}}
\def\ba{\begin{array}}
\def\ea{\end{array}}
\def\beast{\begin{eqnarray*}}
\def\eeast{\end{eqnarray*}}
\def\ben{\begin{enumerate}}
\def\een{\end{enumerate}}

\def\bi{\bibitem}

\begin{document}
\begin{center}
\Large{On Pimsner Popa bases}\\
\bigskip
\large{Keshab Chandra Bakshi}\\
The Institute of Mathematical Sciences\\
Taramani, Chennai 600113,\\
INDIA\\
email: keshabcb@imsc.res.in
\end{center}

\begin{abstract}
 In this paper we examine bases for finite index inclusion of $II_1$ factors and connected inclusion of finite dimensional $C^*$- algebras. 
 These bases behave nicely with respect to basic construction towers. As applications we have 
 studied automorphisms of the hyperfinite $II_1$ factor $R$ which are `compatible with respect to the Jones' tower of finite 
 dimensional $C^*$-algebras'. As a further application, in both Cases we obtain a characterization, in terms of bases, of basic constructions.  
 Finally we use these bases to describe the phenomenon of multistep basic constructions (in both the Cases). 
 
Pimsner-Popa bases.
 
\end{abstract}

\section{Introduction}
We write $(N\subseteq M, tr)$ to denote a unital inclusion of finite von Neumann algebras, with `$tr$' a faithful normal tracial state, and write $N \subset M \stackrel{e_1}{\subset} M_1$ for Jones'  resulting basic construction. The trace  $tr$ is called a {\em Markov trace of modulus $\tau$ } if it extends to a positive trace $Tr:M_1 \longmapsto \C $ such that  $Tr(xe_n)= \tau tr(x)$ for $x \in M$.

 We confine ourselves to  two Cases: (1) when the inclusion is one of $II_1$ factors with finite index, i.e., $[M:N]< \infty$; and (2) when we have a connected inclusion of finite dimensional $C^*$- algebras.
Then it is known that in both Cases ((1) and (2)) there exists a unique  Markov trace on $M$, and we can iterate the basic construction to obtain a tower,$$M_1\subseteq M_2 \subseteq ..\subseteq M_n \subseteq M_{n+1}....$$
 where $M_{n+1} = \langle M_n,e_{n+1} \rangle $ is the result of applying the basic construction for the pair $M_{n-1}\subseteq M_n$ and $e_{n+1}$ is the projection implementing the $tr_{M_n}$ preserving
 conditional expectation of $M_n$ onto $M_{n-1}$. We then obtain a $II_1$ factor $M_\infty$ in both the Cases, 
 which is hyperfinite in Case (2). 
 
 Pimsner and Popa have
 shown (in \cite{PP1}) that for an inclusion $N \subset M$ of $II_1$ factors,
 $M$ is a finitely generated projective module over $N$ if and only if  $[M:N]$ is finite by constructing a family $\{m_j:1 \leq j \leq {n+1}\}$ of elements in $M$, with $n$ equal to the integer part of
 $[M:N]$, which they called ``orthonormal basis'' for the pair $N \subseteq M$. In a similar manner, we find a slightly less restrictive notion of basis in \cite{JS}.
 
  In this paper (in section 2) we see that this notion of basis in \cite{JS} can also  be carried out  in our Case (2) of connected inclusions of finite dimensional $C^*$- algebras. Further  in section 2
 we characterize bases, in both Cases (1) and (2), by three 
 equivalent conditions. One advantage of this characterization is a transparent proof of Corollary \ref{bla1}. This result has been mentioned for the Case of $II_1$ factors in\cite{JS} (Lemma 4.3.4 (i)), but the proof there seems incomplete. Our characterization of bases now clarifies this point, and also shows that bases behave in a nice way with respect to the Jones' tower.
 \par As an application  we show (in 3.1) how the use of bases leads to  a  natural proof of existence, in Case (2), (see \cite{BUR}(Theorem 2.1)) of a unique  extension of an automorphism on $M$ which leaves  $N$  globally invariant,  to an automorphism on the hyperfinite $II_1$ factor $M_\infty$ which is compatible with the tower in the sense of fixing the Jones projections. It has been also proved that the initial automorphism will be automatically trace-preserving.
 \par In \cite{PP2}(Proposition 1.2) Pimsner and Popa have characterized basic construction for $II_1$ factor inclusion  in two equivalent ways. See also \cite{jo}(section 5). 
 In this paper we have characterized basic construction in terms of basis we introduced(Lemma \ref{fvrt}). We have succeeded to obtain a simple characterization of $M_1$  for finite demensional $C^*$- algebra 
 Case also. In \cite{PP2}(Theorem 2.6)Pimsner and Popa have used their characterization of basic construction to describe the $k$-th step of the basic construction. In the section 3.2 we have also given another proof of this construction using our characterization of basic construction and have also done the same for connected inclusion of finite dimensional $C^*$-algebras.

\section{Bases} As stated in the Introduction, we assume $N\subseteq M $ is a unital inclusion of finite 
von Neumann algebras of one of the following two types.\vspace{2mm}\par $\mathcal Case(1)$: 
$N$ and $M$ are $II_1 $ factors with finite index $[M:N]$  and  hence 
there exists unique  Markov trace $tr$ on $M$ of modulus $\tau$ where $ \tau =[M:N]^{-1}$.\vspace{2mm} \par $\mathcal Case(2)$:
Let $N\subseteq M$ be a connected inclusion of finite dimensional 
$C^*$ -algebras and hence there exists unique Markov trace $tr$ on M  of modulus $\tau$
where $\tau={\|G\|}^{-2}$ where $G$  is the  inclusion matrix  for $ N\subseteq M $.

\bigskip For both the Cases the following easy but very useful Lemma holds whose proof can be found in \cite{PP1},
(Lemma 1.2) and for Case(2) see \cite{JS} (Remark 4.3.2(a)).
 \begin{lemma}
   \label{imp}
   If $x_1 \in M_1,$ then there exists unique element $x_0 \in M$ such that $x_1 e_1= x_0e_1$,this element is 
   given by $x_0 = {\tau}^{-1}E_M(x_1e_1)$.
 \end{lemma}
In the following theorem we give three equivalent descriptions of basis, not necessarily orthonormal in the sense of Pimsner-Popa.
\begin{theorem}
   \label{pimsnerpopa}
Let N and M be as in Case(1) or in Case(2). Then for a 
finite set $ \{\lambda_i:i\in I={1,2,...n}\} \subseteq M $, the following 
are equivalent:\vspace{1.9mm}
\\(1) Let $ E_N $ be the $tr$- preserving conditional 
expectation of $M$ onto $N$ and define a matrix $Q$ whose $(i,j)$  entry is given by $q_{ij}=
E_N(\lambda_i {\lambda_j}^*)$.\\Then $Q$ is a projection in $M_n(N) $ such that 
$ tr_{M_n(N)}(Q)={\tau}^{-1}/n$.
\vspace{2 mm}\\$(2)\sum_{i=1}^n{{\lambda_i}^*e_1
{\lambda_i}}=1$, where $ e_1$  is the Jones projection.
\vspace{2mm}\\(3) For any 
$x\in M $ , $x=\sum_{i=1}^n E_N(x{\lambda_i}^*) \lambda_i$.
\end{theorem}
\begin{proof}
 $ (1)\Longrightarrow (2) :$ This proof is mainly inspired by \cite{PP1}. Assume (1) holds. Since $tr$ on $M$ is Markov, it extends to a unique trace on $M_1$, 
 namely $tr_{M_1}$. Put $v_i= e_1 \lambda_i$ 
and $$v=\begin{bmatrix}
         v_{1} & 0 & \dots & 0\\
         v_{2} & 0 &\dots & 0\\
         \vdots & \vdots & \ddots & \vdots\\
         v_{n} & 0 & \dots & 0 
        \end{bmatrix} ~.$$ 
Then, $v_i {v_j}^*=e_1 \lambda_i {\lambda_j}^*
 e_1=E_N(\lambda_i {\lambda_j}^*) e_1=q_{ij} e_1$. Thus,\\
 $$vv^*=
 \begin{bmatrix}
         q_{11}e_1  & q_{12}e_1 & \dots & q_{1n}e_1\\
         q_{21}e_1  & q_{22}e_1 & \dots & q_{2n}e_1\\
         \vdots     &\vdots     & \ddots  & \vdots\\
         q_{n1}e_1  & q_{n2}e_1 & \dots   & q_{nn}e_1
 \end{bmatrix} =Q E $$ where 
                            $$ E=
                             \begin{bmatrix}
                              e_1 & 0 &\dots & 0\\
                              0 & e_1 &\dots & 0\\
                              \vdots & \vdots &\ddots &\vdots\\
                              0 & 0 &\dots & e_1
                              \end{bmatrix}~.$$
Thus by property of Jones projection \cite{JON} (Proposition 3.1.4), $vv^*=QE = EQ$ and hence $ v $ is a partial isometry. Thus 
 $v^*v$ is a projection; i.e., $\sum_iv_i^*v_i$ is a projection $f$ (say)
in $\langle M, e_1 \rangle = M_1$. But $f =\sum_i{\lambda_i}^*e_1 \lambda_i$ satisfies the following equations :
\begin{align*}
tr_{M_1}f & =n~tr_{M_n(M_1)}~(vv^*)\\
& = n~tr_{M_n(M_1)}(QE)\\
 &= n~(1/n)\sum_itr_{M_1}(q_{ii}e_1)\\
 & = \sum_i\tau~ tr(q_{ii}) ~~~~~~~~~~~ (\textrm{Markov property})\\
 & = \tau ~n~  tr_{M_n(N)}(Q)\\
 &=1 ~~~~~~~~~~~~~~~~~~~~~~~~~~~~~~~~(\textrm{by}~(1)).
\end{align*}
Thus $(1-f)\geq 0$ and $tr_{M_1}(f) =1$. Then faithfulness of $tr $
implies $f =1.$
 So $ \sum_i{\lambda_i}^*e_1\lambda_i =1 $. Thus (1) implies (2).\vspace{2mm}

$\\(2)\Longrightarrow (3) $: We assume that (2) holds. Let $x^*\in M $, then 
\begin{align*}
x^*e_1 & = (\sum_i{\lambda_i}^*e_1\lambda_i)x^*e_1\\
 & = \sum_i{\lambda_i}^*E_N(\lambda_ix^*)e_1\\
 & = (\sum_i{\lambda_i}^*E_N(\lambda_ix^*))e_1.
\end{align*}

Again applying Lemma \ref{imp} and then taking adjoint we get (3).

$ \\(3)\Longrightarrow (2):$ We assume (3). Let $x$ and $y$ be two arbitrary elements of $M$. Then,
\begin{align*}
 (\sum_i{\lambda_i}^*e_1\lambda_i)(xe_1y) & =\sum_i{\lambda_i}^*e_1\lambda_ixe_1y\\
                                     &   =\sum_i{\lambda_i}^*E_N(\lambda_i x)e_1y\\
                                     &   =(xe_1y) ~~~~~~~~~~~~(\textrm{by  (3)}).
\end{align*}
Similarly,
\begin{align*}
 (xe_1y)(\sum_i{\lambda_i}^*e_1\lambda_i)& = \sum_i xe_1y{\lambda_i}^*e_1\lambda_i\\
                                         & = \sum_i xe_1E_N(y{\lambda_i}^*)\lambda_i\\
                                         & = (xe_1y)~~~~~~~~~~~~~(\textrm{by  (3)}).
\end{align*}

Then we know the space $Me_1M$, which is linear span of $\{xe_1y:x,y\in M\},$
is a strongly dense *-subalgebra of $M_1$, see for instance \cite{GHJ}(Proposition 3.6.1(vii)). Then since multiplication is separately strongly continuous it follows that
$\sum_i{\lambda_i}^*e_1\lambda_i=1$.

$(2)\Longrightarrow(1):$ Suppose (2) is true. Then,
 \begin{align*}
  e_1(\sum_kq_{ik}q_{kj}) &= e_1(\sum_kE_N(\lambda_i{\lambda_k}^*)E_N(\lambda_k{\lambda_j}^*))\\
  &   =e_1(\sum_kE_N(\lambda_i{\lambda_k}^*E_N(\lambda_k{\lambda_j}^*)))\\
  & =  \sum_k e_1 \lambda_i{\lambda_k}^*E_N(\lambda_k{\lambda_j}^*)e_1\\
  &=  \sum_k e_1\lambda_i{\lambda_k}^*e_1\lambda_k {\lambda_j}^*e_1\\
  & = e_1 \lambda_i(\sum_k{\lambda_k}^*e_1\lambda_k){\lambda_j}^*e_1\\
  & = e_1 \lambda_i {\lambda_j}^*e_1~~~~~~~~~~~~~~~~~(\textrm{by (2)})\\
  & = e_1E_N(\lambda_i{\lambda_j}^*)\\
  & = e_1 q_{ij}.
 \end{align*}
Thus applying Lemma \ref{imp} we get $Q^2 = Q$. Clearly $Q^* = Q$. Hence $Q$ is a projection in 
$M_n(N)$. Now
      \begin{align*}
	tr_{M_n(N)}(Q) & =(1/n)\sum_i tr( q_{ii})\\
	  & = (1/n)\sum_i tr(E_N(\lambda_i{\lambda_i}^*))\\
	  &=  (1/n)\sum_i tr(\lambda_i{\lambda_i}^*)\\
	  & = ({\tau}^{-1}/n )
	  \sum_i tr(e_1 \lambda_i{\lambda_i}^*) ~~\textrm{(Markov Property)} \\
	  & = ({\tau}^{-1}/n)\sum_i tr({\lambda_i}^*e_1 \lambda_i)\\
	  & = ({\tau}^{-1}/n).
	\end{align*} Hence (2) implies (1).
	
	\end{proof}
\begin{remark}	
 Taking adjoints in (3) it follows that the above three 	
are also equivalent to $ x = \sum_{i=1}^n{\lambda_i}^*E_N(\lambda_i x)$, for all
$ x\in M$.
\end{remark}

\begin{definition} 
A finite set $ \{\lambda_i:i\in I\} \subset M$
satisfying any one of the equivalent conditions (i)-(iii) of Theorem \ref{pimsnerpopa} will simply be  called a {\bf basis} for
$M/N$. 
\end{definition}

\noindent \textbf{Existence of bases}: For Case(1) an explicit construction has been given in \cite{PP1}(Proposition 1.3)
 while for Case (2) see \cite{JS} (Lemma 5.7.3), and \cite{JP} (Proposition 2.5). For Case(2) see also \cite{EVA}(section 9.4).

\begin{remark}
 Comparing \cite{PP1} (Proposition 1.3(c)(2)) and Theorem\\\ref{pimsnerpopa} we remark that any Pimsner-Popa 
 basis for  $II_1$ factor inclusions is automatically a basis according to our notion. Also,  motivated by 
 \cite{PP1} and \cite{KO}, Watatani has introduced (in the memoir \cite{WAT}) what he calls `quasi-basis for conditional expectation $E$' in a purely algebraic setting. Assuming the existence of quasi-basis
 he developed index for a conditional expectation of index-finite type, called it $Index~ E$, which he shows to be independent of the choice of quasi-basis. He then investigated 
 Jones' index theory in $C^*$-algebra setting. Observe that, Theorem \ref{pimsnerpopa} (1) now says that $Index~ E$ is same as Jones' index for Case (1) and equals to 
 ${\|G\|}^{2}$ for Case (2).
\end{remark}

\begin{remark}
\label{r}
 The row vector $ [E_N(x{\lambda_1}^*),..,E_N(x{\lambda_n}^*)]\in M_{1\times n}(N)Q$ and conversely if $[x_1,..,x_n]
\in M_{1\times n}(N)Q $ satisfies $x=\sum_{i=1}^n x_i\lambda_i $ then $ x_j=E_N(x{\lambda_j}^*)$ for all $j \in I$.
\end{remark}
Exactly the same proof as in \cite{JS} (Proposition 4.3.3(b)(ii)) works.

\begin{corollary}
\label{bla1}
 
Let $N \subseteq M\subseteq P $ be a tower of $ II_1 $ factors with $[P:N]< \infty $ (or
a tower of finite dimensional $C^*$-algebras where the two inclusions are connected
with inclusion matrices $G$ and $H$ respectively). In either Case, let $\{\lambda_i:
1\leq i\leq m\}$ be a basis for $M/N$ and $\{\mu_j: 1\leq j\leq n\}$ be a basis for 
$P/M$, then  $\{\lambda_i \mu_j :1\leq i\leq m,1\leq j\leq n\}$ is a basis for $P/N$.

\end{corollary}
\begin{proof}
 Let $x\in P$ and as $\{\mu_j\}$ is a basis for $P/M$, we get,
$x=\sum_{j=1}^nE_M(x{\mu_j}^*)\mu_j$. Now note $E_M(x{\mu_j}^*) \in M$ and $\{\lambda_i\}
$ is a basis for $M/N$. Now condition(3) of the Theorem \ref{pimsnerpopa}  yields,

\begin{equation*}
E_M(x{\mu_j}^*)=\\
\sum_{i=1}^m
E_N\{E_M(x{\mu_j}^*){\lambda_i}^*\}\lambda_i.                                                            
\end{equation*}
 Thus we get,
                    \begin{align*}
                     x &=\sum_{j=1}^n[\sum_{i=1}^m E_N\{E_M(x{\mu_j}^*){\lambda_i}^*\}\lambda_i]\mu_j\\
                       &= \sum_{i=1}^m\sum_{j=1}^nE_N\{E_M(x{\mu_j}^*{\lambda_i}^*)\}\lambda_i \mu_j.
                    \end{align*}
Thus again applying (3) of the Theorem \ref{pimsnerpopa} we get that\\$\{\lambda_i\mu_j:1\leq i
\leq m,1\leq j\leq n\}$ is a basis for $P/N$.
\end{proof}

\begin{corollary}
\label{bla2}
If $\{\lambda_i: i\in I=\{1,2,...n\}\}$ is a basis for $M/N$, then $ \{{\tau}^{-1/2}e_1\lambda_i\}
$ is a basis for $M_1/M$. Where $N\subseteq M$ is an in inclusion as in Case (1) or Case (2).
\end{corollary}

\begin{proof}
In the Case (1) $ M_1 $ is a $II_1$ factor such that $[M_1:M]=[M:N]< \infty$. In Case (2)
inclusion matrix for $M\subseteq M_1$ is $G^t$ and hence is a connected inclusion. Now in both
the Cases let $e_2$ be the Jones projection for the inclusion $M\subseteq M_1$. Now,
  \begin{align*}
   \sum_{i=1}^n\{\tau^{-1/2}e_1 \lambda_i\}^*e_2\{\tau^{-1/2}e_1\lambda_i\} &=
        {\tau}^{-1}\sum_{i=1}^n{\lambda_i}^*e_1e_2e_1\lambda_i\\
        &={\tau}^{-1}\tau\sum_{i=1}^n{\lambda_i}^*e_1\lambda_i\\
        &=1~~~  (\textrm{since} ~~\{\lambda_i\}~~ \textrm{is a basis}).
  \end{align*}
Now (2) of the Theorem \ref{pimsnerpopa} yields the result. 
\end{proof}

\begin{remark}\label{well-def}
From Corollary \ref{bla2} every element of $M_1$ is expressible in the form $\sum_{i=1}^{n} x_ie_1y_i$ for some $x_i, y_i \in M$ (in fact, 
$x = \tau^{-1}\sum_{i=1}^n E_M(x\lambda_i^*e_1 ) e_1\lambda_i$); however this does not allow us to define a *-homomorphism on $M_1$ by merely specifying the image of an element of the form $xe_1y$, as we will need to verify that such a `definition' is unambiguous; but we may define the above {\it canonical} decomposition to unambiguously define maps on $M_1$ once we know where to map elements of $N$, the basis vectors $\lambda_i$ and $e_1$. This problem of ambiguity was part of the reason for us the study this notion of bases. The reader need only compare the crisp clarity of the proofs of unambiguity in the definition of $\alpha_1$ in Theorem \ref{keshab} and of $\phi$ in Lemma \ref{fvrt} with the corresponding proofs of Theorem 2.1 in \cite{BUR} (actually only to be found in the arXiv version) and of Proposition 1.2 in \cite{PP2}, to appreciate this remark.
\end{remark}

\begin{corollary}
\label{bla}
Let $N\subseteq M$ as in Case (1) or (2) and $\{\lambda_i:i\in I\}$  be a basis for
$M/N$. Define $\widehat{i(k)} = (i_1,i_2,......i_k)\in I^k,k\geqslant 1$ and 
$\lambda_{\widehat{i(k)}}= {\tau}^{-k(k-1)/4}\lambda_{i_1}e_1\lambda_{i_2}e_2e_1\lambda_{i_3}
.....\lambda_{i_{k-1}}e_{k-1}.....e_1\lambda_{i_k}$.Then $\{\lambda_{\widehat{i(k)}} :
\widehat{i(k)}\in I^k\}$ is a basis for $M_{k-1}/N$.
\end{corollary}

\begin{proof}
 Clearly the statement is true for $ k=1$ with the understanding that
$M_0=M$. Suppose the statement is true for $k$. Now applying Corollary \ref{bla2} recursively we get 
$M_k/M_{k-1}$ has basis\\
$\{\tau^{-k/2} e_ke_{k-1}..e_1\lambda_{i_{k+1}}:i_{k+1}\in I\}$. Then applying Corollary \ref{bla1}
we see that $M_k/N$ has basis,\vspace{1mm}\\
$\{\tau^{-k(k-1)/4}\tau^{-k/2}\lambda_{i_1}e_1\lambda_{i_2}e_2e_1\lambda_{i_3}..\lambda_{i_{k-1}}e_{k-1}..e_1
\lambda_{i_k}e_ke_{k-1}..e_1\lambda_{i_{k+1}}\}$ \\$= \{\lambda_{\widehat{i(k+1)}}:\widehat{i(k+1)}\in I^{k+1}\}$;
and the proof of the inductive step is  complete.
\end{proof}

\section{Applications}
\subsection{Compatible automorphisms of \\the Hyperfinite $II_1$ factor}
 \par Consider an inclusion as in Case (2). Then, we  have a unique Markov trace $tr$ on M.
Next consider the Jones tower $N\subseteq M\subseteq M_1\subseteq M_2\subseteq.....$
and let $R$ be the hyperfinite $II_1$ factor arising from this tower\cite{JS}.
Suppose further we have an automorphism $\alpha_0$ on $M$ such that $\alpha_0(N)=N$.
In this present section we shall show using our concept of  basis how we can costruct a
unique extension of $\alpha_0$ to an automorphism $\alpha$ of the hyperfinite $II_1$ factor $R$ which is compatible with respect to 
the tower in the sense of fixing all the Jones projections and leaving $M_i$ invariant.\par This can be thought of as the finite dimensional $C^*$-algebraic version of \cite{LOI} (Lemma 5.1). In that 
paper Loi studied automorphisms for a pair of factors using standard form of von Neumann algebras, whereas our treatment is based on basis for the corresponding
inclusion. In the similar direction in \cite{KAW} the author has dealt with automorphisms commuting with a faithful normal conditional expectation for a pair 
of $\sigma$-finite von Neumann algebras and related this with an action of a locally compact abelian group. See also \cite{SV}, where the author was more 
concerned with commuting squares.

\begin{lemma}
\label{kod}
Let $N,M,tr,\alpha_0$ be as above.Then $\alpha_0$ 
is automatically trace preserving, that is $tr\circ {\alpha_0}=tr$.
\end{lemma}

\begin{proof}
Let the minimal central projections in $N$ be $\{p_1,p_2,.....p_m\}$ and those in $M$ be 
$\{q_1,q_2,....q_n\}$. Then the inclusion matrix $G$ is an $m\times n$ matrix. Observe $\alpha_0$ 
permutes $p_i$'s and $q_j$'s. Say $p_i\mapsto p_{\tau(i)}$
and $q_j\mapsto q_{\sigma_(j)}$ for $\tau\in{\Sigma}_m$ and $\sigma\in{\Sigma}_n$. As $\alpha_0$
is an automorphism,$G(i,j)=G(\tau(i),\sigma(j))$. Equivalently, $G=TGS$ for permutation matrices $T$
and $S$ of sizes $m$ and $n$ respectively. Let $\vec{t}$ be the trace vector corresponding 
to $tr$ for $M$. Then  it is the unique positive  Perron-Frobenius eigenvector of $G^{t}G$, hence 
also of $S^{-1}G^{t}GS$. But that implies $S\vec{t}$ is a positive eigenvector of $G^{t}G$ with the 
same eigenvalue as of $\vec{t}$ and by uniqueness of Perron -Frobenius theory(see chapter XIII\cite{GAN})we get
$S\vec{t}=\vec{t}$. Hence $tr\circ{\alpha_0}=tr$.
\end{proof}
\par The above proof is due to Vijay Kodiyalam. I sincerely thank him for this.
\begin{theorem}
\label{keshab}
Let $\alpha_0$ be an automorphism of $M$ such that $\alpha_0(N)=N$. Then there is a unique (trace preserving)
automorphism $\alpha_1$  of $M_1$ such that $\alpha_1(e_1)=e_1$, $\alpha_1(M)=M$ and the
restriction  of $\alpha_1$ to $M$  is $\alpha_0$.
\end{theorem}

\begin{proof}
We know there is a basis for $M/N$. Fix such a basis $\{\lambda_i:i\in I\}$. Then we show $\{{\alpha_0(\lambda_i)}
:i\in I\}$ is  also a basis for $M/N$. Let $Q_1$ be the matrix with $(i,j)$ entry given by 
$q_1(i,j)=E_N\{\alpha_0(\lambda_i{\lambda_j}^*)\}$.\\Now, 
   \begin{align*}
      tr_{M_n(N)}(Q_1) &= (1/n)\sum_{i\in I}tr(q_1(i,i))\\
      & =  (1/n)\sum_{i\in I}tr[E_N\{\alpha_0(\lambda_i{\lambda_j}^*)\}]\\
      & = (1/n)\sum_{i\in I} tr (\lambda_i {\lambda_j}^*) & ~~{\rm (by ~Lemma~ (\ref{kod})})\\
      & = tr_{M_n(N)}Q.
    \end{align*}  
 Thus it follows from the Theorem  \ref{pimsnerpopa}, that $\{\alpha_0(\lambda_i)\}$ is a basis for $M/N$.
 Observe since $\alpha_0$ leaves $N$ invariant it follows that $E_N(\alpha_0(x))=\alpha_0(E_N(x))$.
Let  $x\in M_1$, Corollary \ref{bla2} then implies \\
      \begin{equation*}
       x=\sum_{i\in I}{\tau}^{-1}E_M(x{\lambda_i}^*e_1)e_1\lambda_i.
      \end{equation*}

Then define,
        
\begin{equation*}
 \alpha_1(x)={\tau}^{-1}\sum_{i\in I}\alpha_0(E_M(x{\lambda_i}^*e_1))e_1\alpha_0(\lambda_i).
\end{equation*}
There is clearly no ambiguity in the definition of $\alpha_1$.
 Next we show  that $\alpha_1$ is a homomorphism. Consider $y\in M_1$.
 Now using the properties of Jones' projection and the fact that $\alpha_0$ is a homomorphism we get the following series of equations:
 \begin{align*}
      &\alpha_1 (x) \alpha_1 (y)\\
      &\qquad ={\tau}^{-2}\sum_{i,j}\alpha_0[E_M(x{\lambda_i}^*e_1)]e_1
        \alpha_0 (\lambda_i) \alpha_0[E_M(y{\lambda_j}^*e_1)]e_1\alpha_0(\lambda_j)\\
      &\qquad= {\tau}^{-2} \sum_{i,j}\alpha_0[E_M(x{\lambda_i}^*e_1)]
	   E_N(\alpha_0[\lambda_i E_M(y{\lambda_j}^*e_1)])e_1 \alpha_0(\lambda_j)\\
      &\qquad={\tau}^{-2}\sum_{i,j}\alpha_0
	   [E_M(x{\lambda_i}^*e_1)E_N(\lambda_iE_M(y{\lambda_j}^*e_1))]e_1\alpha_0(\lambda_j)\\
	   & ~~~~~~~~~~~\hfill (\text{since $\alpha_0$ and $E_N$ commute})\\
      &\qquad ={\tau}^{-2}\sum_{i,j}\alpha_0[E_M\{x{\lambda_i}^*e_1E_N(\lambda_iE_M(y{\lambda_j}^*e_1))\}]e_1 \alpha_0(\lambda_j)\\
      &\qquad = {\tau}^{-2}\sum_j\alpha_0[E_M\{xE_M(y{\lambda_j}^*e_1)e_1\}]e_1\alpha_0(\lambda_j)~~~~~~~~~~~~~~~~~~~~~(1)\\
      &\qquad\qquad\qquad\qquad\qquad(\text{since $\sum_i{\lambda_i}^*e_1\lambda_i=1$}).
  \end{align*}
Similarly,
\begin{align*}
  &\alpha_1(xy) \\
  &\qquad = {\tau}^{-1}\sum_i\alpha_0[E_M(xy{\lambda_i}^*e_1)]e_1\alpha_0(\lambda_i)\\
  & \qquad= {\tau}^{-1}\sum_i\alpha_0[E_M\{x(\sum_j{\tau}^{-1}E_M(y{\lambda_j}^*e_1)e_1\lambda_j){\lambda_i}^*e_1\}]e_1\alpha_0(\lambda_i)\\
  & \qquad= {\tau}^{-2}\sum_{i,j}\alpha_0[E_M\{xE_M(y{\lambda_j}^*e_1)E_N(\lambda_j{\lambda_i}^*)e_1\}]e_1\alpha_0(\lambda_i)\\
  & \qquad= {\tau}^{-2}\sum_{i,j}\alpha_0[E_M\{xE_M(y{\lambda_j}^*e_1E_N(\lambda_j{\lambda_i}^*))e_1\}]e_1\alpha_0(\lambda_i)\\
  & \qquad= {\tau}^{-2}\sum_i\alpha_0[E_M\{xE_M(y{\lambda_i}^*e_1)e_1\}]e_1\alpha_0(\lambda_i)~~~~~~~~~~~~~~~~~~~~~(2)\\
  &\qquad\qquad\qquad\qquad\text{(since $\sum_j{\lambda_j}^*e_1\lambda_j=1$)}.
\end{align*}
Now comparing equations (1) and (2) we conclude that
$\alpha_1$ is indeed a homomorphism. Next we show $\alpha_1$ fixes $e_1$. Observe,
\begin{equation*}
 e_1={\tau}^{-1}\sum_iE_M(e_1{\lambda_i}^*e_1)e_1\lambda_i.
\end{equation*}
Now using our definition of $\alpha_1$ and property of Jones' projection it is easy to see that,
\begin{align*}
  \alpha_1(e_1) & = {\tau}^{-1}\sum_i\alpha_0[E_M\{E_N({\lambda_i}^*)e_1\}]e_1\alpha_0(\lambda_i)\\ 
  & = \sum_i \alpha_0(E_N({\lambda_i}^*))e_1\alpha_0(\lambda_i)\\
  & = \sum_iE_N(\alpha_0({\lambda_i}^*))e_1 \alpha_0(\lambda_i)~~~~~~~(\textrm {as}~ E_N~ \textrm{and}~ \alpha_0~\textrm {commute})\\
  & = \sum_ie_1 \alpha_0({\lambda_i})^*e_1\alpha_0(\lambda_i)\\
  & = e_1.
\end{align*}
In the last equation we have  used the fact that $\{\alpha_0(\lambda_i)\}$ is a basis for $M/N$.\par
Next we will show that $\alpha_1$ agrees with $\alpha_0$ when it is restricted to $M$.
Now, since $\alpha_0$ is a automorphism for $x\in M$ we find that,
\begin{align*}
 \alpha_1(x) & = {\tau}^{-1}\sum_i\alpha_0\{E_M(x{\lambda_i}^*e_1\}e_1\alpha_0(\lambda_i)\\
  & = \sum_i\alpha_0(x{\lambda_i}^*)E_M(e_1)e_1\alpha_0(\lambda_i)\\
  & = \sum_i\alpha_0(x){\alpha_0(\lambda_i)}^*e_1\alpha_0(\lambda_i)~~~~~~(\textrm {since}~ E_M(e_1)=\tau)\\
  & = \alpha_0(x)~~~~~~~~~~~~~~~(\textrm {as}~\{\alpha_0(\lambda_i)\}~\textrm {is a basis for}~M/N).\\
\end{align*}
Now we want to show that $\alpha_1$ is onto.\par Let $y\in M_1$. Then, $y=\sum_i y_ie_1\alpha_0(\lambda_i)$, since ${\alpha_0(\lambda_i)}$ is a basis for $M/N$.
As, $\alpha_0$ is an automorphism there is a unique $x_i\in M$ such that $\alpha_0(x_i)= y_i$. Put $x= \sum_ix_ie_1 \lambda_i$. Then $x$ belongs to $ M_1$.
Now as we have already proved that $\alpha_1$ is a homomrphism which preserves $e_1$ and agree with  $\alpha_0$ when restricted to $M$ it follows trivially that 
$\alpha_1(x)=y$. Thus $\alpha_1$ is onto.\par Lastly we show $\alpha_1 $ is one-one. Observe, $\alpha_1$ is *-preserving, since
if $x=\sum_ix_ie_1\lambda_i$ we find, exactly as above, that
\begin{equation*}
\alpha_1(x^*)=\sum_i{\alpha_1(\lambda_i)}^*e_1{\alpha_1(x_i)}^*=\{\sum_i\alpha_1(x_i)e_1\alpha_1(\lambda_i)\}^*={\alpha_1(x)}^*.
\end{equation*}
Now,
\begin{align*}
 tr(\alpha_1(x)) & = tr(\sum_i\alpha_0(x_i)e_1\alpha_0(\lambda_i))\\
  & = \sum_i tr \{e_1\alpha_0(\lambda_i) \alpha_0(x_i)\}\\
  & = \tau\sum_i tr\{\alpha_0(\lambda_ix_i)\} ~~~~(\textrm{Markov property})\\
  &= \tau\sum_itr(\lambda_ix_i)~~~~~~~~~~~(\textrm{by~Lemma}~\ref{kod})\\
  & = \sum_itr(x_ie_1\lambda_i)~~~~~~~~~~~(\textrm{Markov~property})\\
  &= tr(x).
 \end{align*}
so $\alpha_1$ is $tr$-preserving and hence one-one. The uniqueness assertion is obvious since $M$ and $e_1$ generate $M_1$. Thus $\alpha_1$ 
satisfies all the properties mentioned in the Theorem.
\end{proof}
\begin{corollary}
 Let $\alpha_0$ be as in the previous theorem. Then there is a unique (trace preserving) automorphism $\alpha$ of the hyperfinite $II_1 $ factor $R$
 such that $\alpha(e_i)=e_i,\alpha(M_i)=M_i$  for all $i\geq 1$ and ${\alpha}|M=\alpha_0 $.
\end{corollary}
\begin{proof}
 Apply Theorem \ref{keshab} recursively for the tower of basic construction to get a unique (trace preserving)
 automorphism $\alpha_i$ on $M_i$ which leaves $M_j$ invariant and fixes
 all $e_j$ such that $1\leq j\leq i$ and ${\alpha_i}|M_j=\alpha_j$. Thus we can  
 define an automorphism(compatible with respect to the tower) $\alpha_{\infty}$ on ${\cup}_iM_i$ by,
 $\alpha_{\infty}(x)=\alpha_j(x) $ for $x\in M_j$. Now as $\alpha_{\infty}$ is bounded it extends to
 trace preserving  automorphism $\alpha$ (say) on $R.$ Also since
  $M_1$ and $e_i$ s  generate $R$ uniqueness is straightforward.
\end{proof}

\subsection{Iterating basic construction}
The following gives a characterization of basic construction  using bases, in both  Case(1) and Case(2). 
This would be needed for our proof of the assertion regarding $k$-th step basic constructions.
\begin{lemma}
\label{fvrt}
Let $N\subseteq M$ be as in Case(1) or Case(2).
Assume $\{\lambda_i:i\in\{1,2,..n\}\}$
is a basis for $M/N$(which exists in both the Cases). Let $P$ be a $II_1$ factor in Case(1) or 
a finite dimensional $C^*$-algebra in Case (2) such that $P$ contains $M$  and also  contains a projection $f$ such that
$\sum_{i=1}^n{{\lambda_i}^*f{\lambda_i}}=1$ and satisfies
further the following two properties :\par$1)fxf=E_N(x)f$  for all $x\in
M$ and\par2)\{${\tau}^{-1/2}f{\lambda_i}$\} is a basis for $P/M $.
 \\In addition for Case(2) $P$ satisfies the following property also: \par3) $n\longmapsto nf$ is an injective map from $N$ into $P$.\\Then there exists an isomorphism from $M_1=\langle M,e_1 \rangle$ onto $P$ which maps
$e_1$ to f.
\end{lemma}
In this situation  we say that $P$ is an instance of basic construction applied to the inclusion $N\subseteq M$ with 
a choice of projection implementing the conditional expectation being given by $f$.
\begin{proof}
$\mathcal Case 1:$
 Let $x \in M_1$. Now from Corollary \ref{bla2}  it follows that\\
 \begin{equation*}
 x = \sum_i{\tau}^{-1/2}E_M(x{\tau}^{-1/2}{\lambda_i}^*e_1)e_1{\lambda_i}.
 \end{equation*}
 Put $a_i={\tau}^{-1}E_M(x{\lambda_i}^*e_1)$, 
 then define a map $\phi :M_1 \mapsto P$ by $\phi(x) = \sum_ia_if{\lambda_i}$, which is clearly well-defined. 
 Note, if $y= \sum_ib_ie_1\lambda_i$ such that $[b_1,b_2,......b_n] \in M_{1\times n }(M)Q$,  
 then by Remark \ref{r}  we conclude 
 \begin{equation}
  \phi(y)=\sum_ib_if\lambda_i.
 \end{equation}
 Since, if $Q_1$ is the matrix whose $i-j$ th entry is given by $q_1(i,j)=E_M(({\tau}^{-1/2}e_1\lambda_i)({{\tau}^{-1/2}e_1\lambda_j})^*)$,
then $q_1({i,j})=E_N(\lambda_i{\lambda_j}^*)=q_{ij}$.
Now, let $x$ be as above and let $y \in M_1$. Put $y=\sum_ib_ie_1\lambda_i$ where \\
$b_i= {\tau}^{-1}E_M(y{\lambda_i}^*e_1)$. Then the following equations follow from properties of Jones' projection,
\begin{align*}
 \phi(xy) & = \phi\{\sum_{i,j} a_iE_N(\lambda_ib_j)e_1\lambda_j\}\\
    &= \phi\{\sum_{i,j}{\tau}^{-1}E_M(x{\lambda_i}^*e_1)E_N(\lambda_ib_j)e_1\lambda_j\}\\
    &= \phi\{\sum_{i,j}{\tau}^{-1}E_M(x{\lambda_i}^*e_1E_N(\lambda_ib_j))e_1\lambda_j\}\\
    &= \phi\{\sum_j{\tau}^{-1}E_M(xb_je_1)e_1 \lambda_j\}~~~~(\textrm {since}~\sum_i{\lambda_i}^*e_1\lambda_i=1)\\
    &= \phi\{\sum_j{\tau}^{_1}E_M[x{\tau}^{-1}E_M(y{\lambda_j}^*e_1)e_1]e_1\lambda_j\}.
\end{align*}
Now it can be easily checked that,\\
 $[{\tau}^{-1}E_M\{x{\tau}^{-1}E_M(y{\lambda_1}^*e_1)e_1\},{\tau}^{-1}E_M\{x{\tau}^{-1}E_M(y{\lambda_2}^*e_1)e_1\},.....,\\{\tau}^{-1}E_M\{x{\tau}^{-1}E_M(y{\lambda_n}^*e_1)e_1\}] \in M_{1\times n}(M)Q.$

Thus, it follows from equation(3.1) that,
\begin{align*}
\phi(xy) & =\sum_j{\tau}^{-1}E_M[x{{\tau}^{-1}E_M(y{\lambda_j}^*e_1)}e_1]f{\lambda_j}\\
  & = \sum_{i,j}{\tau}^{-1}E_M(a_ie_1{\lambda_i}b_je_1)f{\lambda_j}\\
      &= \sum_{i,j}{\tau}^{-1}a_iE_M(E_N(\lambda_i b_j)e_1)f{\lambda_j}\\
      &=\sum_{i,j}a_iE_N(\lambda_ib_j)f\lambda_j\\
      & = \sum_{i,j} a_if\lambda_i b_j f \lambda_j~~~~~~~~~~(\textrm {by~assumption}~(1))\\
      &=\phi(x)\phi(y).
\end{align*}

Also, we have,
\begin{align*}
 \phi(e_1) & = {\tau}^{-1}\sum_iE_M(e_1{\lambda_i}^*e_1)f{\lambda_i}\\
     & = \sum_iE_N({\lambda_i}^*)f \lambda_i~~~~~~~~~(\textrm {since}~E_M(e_1)=\tau)\\
     & = \sum_i f{\lambda_i}^*f\lambda_i~~~~~~~~~~~~~~(\textrm {by assunption}~(1))\\
     & = f~~~~~~~~~~~~~~(\textrm {since}, ~~\sum_i{\lambda_i}^* f {\lambda_i}= 1).
\end{align*}
Thus $\phi$ is a nonzero homomorphism.
Now assume, $x\in M$, then,
\begin{align*}
 \phi(x) & = \sum_i{\tau}^{-1}E_M(x{\lambda_i}^*e_1)f {\lambda_i}\\
    & = \sum_i x{\lambda_i}^*f{\lambda_i}~~~~~~~~~(\textrm {since}~x{\lambda_i}^*\in M)\\
    & = x.
\end{align*}
$\phi$ is also *-preserving, as, if $x= \sum_ia_ie_1\lambda_i$ is  any element of $M_1$, then the  following identities hold:
\begin{align*}
 \phi(\{\sum_ia_ie_1\lambda_i\}^*) & = \sum_i{\lambda_i}^*f{a_i}^*~~~~(\textrm{since}~\phi(e_1)= f~\textrm {and}~ {\phi}|_M = id)\\
     & = \{\sum_ia_i f\lambda_i\}^*\\
     & = \{\phi(\sum_ia_i e_1 \lambda_i)\}^*.
\end{align*}

 Thus $\phi({x}^*)= {\phi(x)}^*$.\\Since we are now in a factor $\phi$ is automatically injective.\\Finally we show $\phi$ is onto. 
 For this purpose assume $z\in P$,  assumption(2) then implies 
 $z=\sum_ic_if\lambda_i$ for some $c_i\in M$. Put, $y = \sum_ic_i e_1 \lambda_i $ which belongs to  $M_1$ and since $\phi$
 is a homomorphism sending $e_1$ to $f$ and whose restriction to $M$ is identity, 
 we clearly get $\phi(y) = z$, proving onto.
 Thus $\phi$ is an isomorphism satisfying all the conditions stated in the Lemma.\\ \\ 
$\mathcal Case2:$ Note assumption(2) implies $P= MfM$. Also this together with assumption(1) imply that $PfP=MfM$. Thus $P=PfP$
 which forces $Z_P(f)=1$. Now just applying Corollary 5.3.2 in \cite{JS} we get the result.\par This completes the Lemma.
 
\end{proof}
Now we give another proof of $k$-th step basic construction for an inclusion of $II_1$ factors using basis and also we show it can be done for Case(2).
\begin{theorem}
 Let $N\subseteq M$ be a pair of von Neumann algebras as in Case(1) or (2) and $N\subseteq M\subseteq M_1\subseteq.....$ be 
 the tower of $II_1$ factors (or finite dimensional $C^*$-algebras)in Case(1) 
 (or in Case(2) respectively) which can be obtained by iterating basic construction. Let $e_i\in M_i$ be the Jones' projections.
 Then  for $m\geq 0,k\geq -1$, 
 $M_k\subseteq M_{k+m}\subseteq M_{k+2m} $ is an instance of basic construction
 with a choice of projection implementing  the conditional expectation of $M_{k+m}$ onto $M_k$ is given by\\
\begin{align*}
& e_{[k,k+m]}\\ 
&\qquad={\tau}^{-m(m-1)/2}(e_{k+m+1}e_{k+m}...e_{k+2})(e_{k+m+2} e_{k+m+1}...e_{k+3})\\
&\qquad\qquad...(e_{k+2m}e_{k+2m-1}....e_{k+m+1}).
\end{align*} 
\end{theorem}
\begin{proof}
Without loss of generality we shall prove that $M_{-1}\subseteq M_n\subseteq M_{2n+1}$ 
is an instance of basic construction with $e_{[-1, n]}$ is the required projection. Assume $\{\lambda_i:i\in{1,2,..n}\}$
is a basis for $M/N$(which exists in both the Cases).
Now from Corollary $\ref{bla}$ we know that $\{\lambda_{\widehat{i(n+1)}}\}$ is a basis 
 for $M_n/N$.

Now applying Corollary \ref{bla1} and Corollary \ref{bla2} repeateadly we get $M_{2n+1}/M_n $ has basis,
$$
{\tau}^{{-1/2}\{(n+1)+(n+2)+..+(2n+1)\}}(e_{n+1}..e_1)\lambda_{i_1}
(e_{n+2}..e_1)\times
$$ 
$$
\lambda_{i_2}.....(e_{2n}..e_1)\lambda_{i_n}(e_{2n+1}..e_1)\lambda_{i_{n+1}}.
$$
Observe that,
$$
 (e_{n+1}...e_1)\lambda_{i_1}(e_{n+2}..e_1)\\
 \lambda_{i_2} ..(e_{2n+1}...e_1)\lambda_{i_{n+1}}
$$
 $$~~~~~~~~~~~~~~~~~~~~~~~~~~~~~=(e_{n+1}..e_1)(e_{n+2}..e_2)(e_{n+3}..e_3)..(e_{2n+1}..e_{n+1})$$
  $$~~~~~~~~~~~~~~~~~\lambda_{i_1}e_1\lambda_{i_2}e_2e_1\lambda_{i_3}..\lambda_{i_n}e_n..e_1\lambda_{i_{n+1}}.
$$
In other words it shows that, $M_{2n+1}/M_n$ has basis as, \\
$ 
\{{\tau}^{-(n+1)/2}e_{[-1,n]}\lambda_{\widehat{i(n+1)}}\}$.\par Note, $[M_n:N]={[M:N]}^{(n+1)}= {\tau}^{-(n+1)}$.
Thus condition (2) of the Lemma \ref{fvrt} holds for factor Case. \par To do the same  for finite 
dimensional $C^*$-algebra we break this into two Cases.\\ 
$\mathcal Case 1:$ Suppose $n$ is odd. Then the inclusion matrix for $N\subseteq M$ would be ${(G{G}^t)}^k$ where $n=(2k-1)$. But it is easy to see that 
$\|{(G{G}^t)}^k\|={\|G\|}^{2k}={\|G\|}^{(n+1)}={\tau}^{-(n+1)/2}$. 
Thus condition (2) of the Lemma \ref{fvrt} holds in this Case.\\ $\mathcal Case 2:$ Here $n$ is even $n=2m$ (say). Then the  inclusion matrix 
for $N\subseteq M$ would be $G{({G}^t G)}^m$. Then we see,
$\|{G}^tG{({G}^tG)}^m\| \leq \|{G}^t\|\|G{({G}^tG)}^m\|= \|G\| \|G{({G}^t G)}^m\|$.
Now applying the Case(1) in left hand side, we get that ${\|G\|}^{2m+1}\leq \|G{({G}^tG)}^m\|$. The opposite inequality 
is obvious. Thus condition(2) of Lemma \ref{fvrt} holds in this Case also.\\We need to show that , for all $k\geq 1$,(for both Cases),
\begin{equation}
\sum_{i_1,i_2,....,i_{k}}{\lambda}^*_{\widehat{i(k)}}e_{[-1,k-1]}\lambda_{\widehat{i(k)}} =1.
\end{equation}
We prove it by induction over $k\geq 1$. It is easy to see that,
$$
\lambda_{\widehat{ i(n)}}({\tau}^{-n/2}e_n..e_1\lambda_{i_{n+1}})= \lambda_{\widehat{i(n+1)}}.$$ and hence, $$
({\tau}^{-n/2}{\lambda}^*_{i_{n+1}}e_1...e_n){\lambda}^*_{\widehat{i(n)}}= {\lambda}^*_{\widehat{i(n+1)}}.
$$Suppose, as induction hypothesis, for $n\geq 1$,
\begin{equation}
 \sum_{i_1,i_2,...,i_n}{\lambda}^*_{\widehat{i(n)}}e_{[-1, n-1]} {\lambda_{(\widehat{i(n)}}} = 1.
\end{equation}
Since $\sum_i{\lambda_i}^*e_1\lambda_i =1 $, we see that equation (3.2) holds for $k= 1$.
Also we know, for $n\geq 1$,
\begin{multline*}
 e_{[-1,n]}={\tau}^{-n}(e_{n+1} e_{n+2}...e_{2n+1})e_{[-1,n-1]}(e_{2n} e_{2n-1}..e_{n+1}).
\end{multline*}
Thus,
\begin{align*}
 &\sum_{i_1,i_2,..,i_{n+1}}{\lambda}^*_{\widehat{i(n+1)}}e_{[-1,n]}\lambda_{\widehat{i(n+1)}}  \\
 &\qquad = \sum_{i_1,i_2,..i_{n+1}}{\tau}^{-2n}{\lambda}^*_{i_{n+1}}(e_1 e_2...e_n){\lambda}^*_{\widehat{i(n)}}(e_{n+1}e_{n+2}..e_{2n+1})e_{[-1,n-1]}\\
 &\qquad\qquad\qquad\qquad(e_{2n}..e_{n+1})\lambda_{\widehat{i(n)}}(e_n e_{n-1}..e_1)\lambda_{i_{n+1}}\\
 &\qquad =\sum_{i_1,i_2,,..i_{n+1}}{\tau}^{-2n}{\lambda}^*_{i_{n+1}}(e_1e_2..e_n)(e_{n+1}..e_{2n+1}){\lambda}^*_{\widehat{i(n)}}e_{[-1,n-1]}
 \lambda_{\widehat{i(n)}}\\
 &\qquad\qquad\qquad\qquad(e_{2n}e_{2n-1}..e_{n+1})(e_n e_{n-1}..e_1)\lambda_{i_{n+1}}\\
 &\qquad = {\tau}^{-2n}\sum_{i_{n+1}}{\lambda}^*_{i_{n+1}}(e_1e_2....e_{2n+1})(e_{2n}e_{2n-1}...e_1) \lambda_{i_{n+1}}\\
 &\qquad\qquad\qquad\qquad\text{[by equation (3.3)]}\\
 &\qquad = \sum_{i_{n+1}}{\lambda}^*_{i_{n+1}}e_1 {\lambda_{i_{n+1}}}\\
 &\qquad\qquad\qquad\qquad\text{[since $(e_1e_2..e_{2n+1})(e_{2n}e_{2n-1}..e_1)= {\tau}^{2n}e_1$]}\\
 &\qquad=1.
\end{align*}
Here, the  second  equation  holds as $\lambda_{\widehat{i(n)}}\in M_{n-1}$ and $(e_{n+1}e_{n+2}..e_{2n+1}),\\(e_{2n}e_{2n-1}..e_{n+1})$ 
both commutes with $M_{n-1}$.
\par Hence the induction is complete.\\Now  we  show  property(1) of the Lemma \ref{fvrt}.\par As induction hypothesis, suppose,for $n\geq 0,$
$$ e_{[-1,n]}x_n e_{[-1,n]} = E_N(x_n)e_{[-1,n]}~ \text{for}~ x_n\in M_n.$$It trivially holds for $n=0$.
 Then, for $n\geq 0$,and for $x_{n+1} \in M_{n+1}$, we get the following array of equations,
\begin{align*}
&e_{[-1,n+1]}x_{n+1}e_{[-1,n+1]}\\
&\qquad= {\tau}^{-2(n+1)}(e_{n+2}..e_{2n+3})e_{[-1,n]}(e_{2n+2}..e_{n+2})\\
&\qquad\qquad\qquad x_{n+1}(e_{n+2}..e_{2n+3})e_{[-1,n]}(e_{2n+2}..e_{n+2})\\
&\qquad= {\tau}^{-2(n+1)}(e_{n+2}..e_{2n+3})e_{[-1,n]}(e_{2n+2}..e_{n+3})\\
&\qquad\qquad\qquad E_{M_n}(x_{n+1})(e_{n+2}..e_{2n+3})e_{[-1,n]}(e_{2n+2}..e_{n+2})\\
&\qquad={\tau}^{-2(n+1)}(e_{n+2}..e_{2n+3})e_{[-1,n]}(e_{2n+2}..e_{n+3})\\
&\qquad\qquad\qquad(e_{n+2}..e_{2n+3})E_{M_n}(x_{n+1})e_{[-1,n]}(e_{2n+2}..e_{n+2})\\
&\qquad ={\tau}^{-2(n+1)}(e_{n+2}..e_{2n+3})e_{[-1,n]}({\tau}^ne_{2n+2}e_{2n+3})\\
&\qquad\qquad\qquad E_{M_n}(x_{n+1})e_{[-1,n]}(e_{2n+2}..e_{n+2})\\
&\qquad= {\tau}^n{\tau}^{-2(n+1)}(e_{n+2}..e_{2n+2})e_{[-1,n]}(e_{2n+3}e_{2n+2}e_{2n+3})\\
&\qquad\qquad\qquad E_{M_n}(x_{n+1})e_{[-1,n]}(e_{2n+2}..e_{n+2})\\
&\qquad= {\tau}^{-(n+1)}(e_{n+2}..e_{2n+3})e_{[-1,n]}E_{M_n}(x_{n+1})e_{[-1,n]}(e_{2n+2}..e_{n+2})\\
&\qquad= {\tau}^{-(n+1)}(e_{n+2}..e_{2n+3})E_N(x_{n+1})e_{[-1,n]}(e_{2n+2}..e_{n+2})\\
&\qquad\qquad\qquad\qquad\text{[Induction hypothesis]}\\
&\qquad=E_N(x_{n+1})e_{[-1,n+1]}.
\end{align*}
The  fourth equation holds because of the almost trivial fact that 
\begin{equation}
(e_{2n+2}..e_{n+3})(e_{n+2}..e_{2n+3}) = {\tau}^ne_{2n+2}e_{2n+3}.
\end{equation}

It should be mentioned that throughout we have used the fact that, for $n\geq0$,
\begin{equation*}
e_{[-1,n+1]}= {\tau}^{-(n+1)}\\(e_{n+2}..e_{2n+3})e_{[-1,n]}(e_{2n+2}..e_{n+2}).
\end{equation*}
This completes the induction.\\Now using Lemma \ref{fvrt} we get the desired result for $II_1$ factor Case.
\\For finite dimensional ${C}^*$- algebra the only remaining thing is to prove that the map $x\longmapsto xe_{[-1,n]}$
for $x\in N$ is injective. From Lemma \ref{imp} it follows that $xe_1=0$ implies $x=0$ for $x\in N$, proving the above fact for $n=0$. 
 Suppose the statement is true for $ (n-1)$, that is for $x\in N, xe_{[-1,n-1]}=0 $
 implies $x=0$. Let for $x\in N,  xe_{[-1,n]}= 0$. Thus,
${({\|xe_{[-1,n]}\|}_2)}^2 =tr(xe_{[-1,n]}x^*)=0$. Note,
\begin{align*}
  0 & = tr(e_{[-1,n]}x^*x) \\
  &= tr((e_{n+1}e_{n+2}..e_{2n+1})e_{[-1,n-1]}(e_{2n}..e_{n+1})x^*x)\\
  & = tr(e_{[-1,n-1]}(e_{2n}..e_{n+1})(e_{n+1}..e_{2n+1})x^*x)~~~~~~~(\textrm{since}~ x^*x\in N)\\
  & = tr(e_{[-1,n-1]}({\tau}^{n-1}e_{2n}e_{2n+1})x^*x).~~~~~~~~~~~(\textrm{by~equation}(3.4))
\end{align*}
But as we know $tr$ is Markov, we conclude from the last equation $tr(e_{[-1,n-1]}x^*x)=0$, that is $tr(xe_{[-1,n-1]}x^*)=0 $. In other words,$xe_{[-1,n-1]}=0$ and now from induction hypothesis we conclude
$x=0$. Hence the induction is complete.\par This completes the proof for both the Cases.
\end{proof}
\section{Acknowledgement} I wish to thank  V. S. Sunder for many helpful discussions.

\bibliographystyle{plain}

\begin{thebibliography}{1}

\bibitem{EVA}
Evans D. E., and Kawahigashi Y.,
\newblock Quantum symmetries on operator algebras,
\newblock {\em Oxford Mathematical Monographs }
The Clarendon Press, Oxford University Press, New York, 1998,
\newblock Oxford Science Publications.

\bibitem{GHJ}
Frederick~M. Goodman, Pierre de~La~Harpe, and Vaughan~F.R. Jones(1989),
\newblock {\em Coxeter graphs and towers of algebras}, volume~14,
\newblock Springer-Verlag New York.

\bibitem{GAN}
Gantmacher F.R.,
\newblock The theory of matrices, vol. 2, Chelsea, New York, 1959,
\newblock {\em Mathematical Reviews (MathSciNet): MR99f}, 15001, 1979.

\bibitem{jo}
 Jolissaint P.,
\newblock Index for pairs of finite von Neumann algebras,
\newblock {\em Pacific Journal of Mathematics},\textbf{146(1)}(1990) 43-70.

\bibitem{JON}
Jones V.F.R.,
\newblock Index for subfactors,
\newblock {\em Inventiones Mathematicae}, \textbf{72(1)}(1983) 1-25.

\bibitem{JP}
 Jones V.F.R. and Penneys D.,
\newblock The embedding theorem for finite depth subfactor planar algebras,
\newblock {\em Quantum Topol.} \textbf{2} Issue (3)(2011) 301-337,


\bibitem{JS}
Jones V.F.R. and  Sunder,V.S.,
\newblock {\em Introduction to subfactors}, volume 234.
\newblock Cambridge University Press, 1997.

\bibitem{KAW}
Kawahigashi Y.,
\newblock Automorphisms commuting with a conditional expectation onto a subfactor with finite index,
\newblock{\em J. Operator Theory},
\textbf {28}(1992) no. 1, 127-145.

\bibitem{KO}
Kosaki H.,
\newblock Extension of Jones' theory on index to arbitrary factors
\newblock {\em Journal of Functional Analysis},
\textbf{66}(1986) 123-140.

\bibitem{LOI}
  Loi P. H.,
\newblock On automorphisms of subfactors,
\newblock {\em Journal of Functional Analysis},
\textbf{141}(2) (1996) 275-293.

\bibitem{PP1}
Pimsner M. and Popa S.,
\newblock Entropy and index for subfactors,
\newblock In {\em Annales scientifiques de l'Ecole normale sup{\'e}rieure},
  \textbf{19} (1986) 57-106.

 \bibitem{PP2}
Pimsner M. and Popa S.,
\newblock Iterating the basic construction,
\newblock {\em Transactions of the American Mathematical Society},
  \textbf{310(1)} (1988) 127--133 .
  
 \bibitem{BUR}
Richard D. Burstein,
\newblock Group-type subfactors and Hadamard matrices,
\newblock {\em Trans. Amer. Math. Soc.}.
\textbf{367}(2015) 6783-6807,
see also \newblock {\em arXiv preprint arXiv:0811.1265}(2008)

\bibitem{SV}
Svendsen A.L.,
\newblock  Automorphisms of subfactors from commuting squares,
\newblock {\em Trans. Amer. Math. Soc.},
\textbf{356}(2004) no.6  2515-2543.


\bibitem{WAT}
Watatani,Y.,
\newblock Index for $C^*$-subalgebras,
\newblock {\em Mem. Amer. Math. Soc.},
\textbf{424}(1990) vi+117 pp.


\end{thebibliography}

\end{document}